\documentclass{amsart}

\newtheorem{thm}{Theorem}
\newtheorem{cor}[thm]{Corollary}
\newtheorem{lem}[thm]{Lemma}
\newtheorem{prop}[thm]{Proposition}

\theoremstyle{remark}
\newtheorem{rem}[thm]{Remark}
\newtheorem{ex}[thm]{Example}

\theoremstyle{definition}
\newtheorem{defin}[thm]{Definition}

\begin{document}

\title[Fractional $h$-difference equations]{Fractional
$h$-difference equations\\
arising from the calculus of variations}

\author[R. A. C. Ferreira]{Rui A. C. Ferreira}

\address{Department of Mathematics,
Faculty of Engineering and Natural Sciences,\newline
Lusophone University of Humanities and Technologies,
1749-024 Lisbon, Portugal}

\email{ruiacferreira@ua.pt}


\author[D. F. M. Torres]{Delfim F. M. Torres}

\address{Department of Mathematics,
University of Aveiro, 3810-193 Aveiro, Portugal}

\email{delfim@ua.pt}


\thanks{Submitted 15-Aug-2010; revised 16-Jan-2011 and 30-Jan-2011;
accepted 31-Jan-2011; for publication in 
\emph{Applicable Analysis and Discrete Mathematics}.}





\begin{abstract}
The recent theory of fractional $h$-difference equations introduced
in [{\sc N. R. O. Bastos, R. A. C. Ferreira, D. F. M. Torres}:
{\it Discrete-time fractional variational problems},
Signal Process. {\bf 91} (2011), no.~3, 513--524],
is enriched with useful tools for the explicit solution
of discrete equations involving left and right
fractional difference operators. New results for
the right fractional $h$ sum are proved. Illustrative examples
show the effectiveness of the obtained results in solving
fractional discrete Euler--Lagrange equations.
\end{abstract}

\maketitle


\section{Introduction}

The fractional calculus is a generalization of (integer order) differential calculus,
allowing to define derivatives (and integrals) of real or complex order \cite{Mill1,samko}.
It is a mathematical subject that has proved to be very useful in applied fields
such as economics, engineering, and physics \cite{Almeida,Magin,M:T,Mach}.
Several definitions of fractional derivatives, including Riemann--Liouville, Caputo,
Riesz, Riesz–-Caputo, Weyl, Grunwald--Letnikov, Hadamard, and Chen derivatives,
are available in the literature \cite{MyID:154,MR2421931,MyID:149,MyID:163}.
The most common used fractional derivative is the
Riemann--Liouville \cite{MR2558546,El-Nab:T,MR2433010,Moz:T}.
Analogously, one can define a discrete fractional derivative in different ways.
In 1989, Miller and Ross introduced the discrete analogue
of the Riemann--Liouville fractional derivative
and proved some properties of the fractional difference operator \cite{Mill0}.
More results on the theory introduced by Miller and Ross are given
in the works of Atici and Eloe \cite{Atici0,Atici}.
See also \cite{withNunoDorota,withNuno:cfc,Goodrich},
and \cite{Atici:SS} for applications to the
Gompertz fractional difference equation and tumor growth models.
Regarding other fractional discrete definitions, we refer
the reader to \cite{Anas,BFT,gray,ort} and references therein.
Here we follow \cite{BFT:SP}, \textrm{i.e.}, we adopt a more general
fractional $h$-difference Riemann--Liouville operator.
The presence of the $h$ parameter is particularly interesting
from the numerical point of view,
because when $h$ tends to zero the solutions
of the fractional difference equations can be seen
as approximations to the solutions of corresponding
Riemann--Liouville fractional differential equations
\cite{BFT:SP,Frederico:Torres1}
(\textrm{cf.} Proposition~\ref{prop:ref:asked}).

In the recent work of Bastos \textrm{et al.} \cite{BFT:SP},
necessary optimality conditions of first and second order
are proved for the fractional $h$-difference variational problem
\begin{equation}
\label{eq:prb:int}
\begin{gathered}
\mathcal{L}(y)=
\sum_{t=\frac{a}{h}}^{\frac{b}{h}-1}
L\left(th,y(\sigma_h(th)),\left({_a}\Delta_h^{\alpha}
y\right)(th),\left({_h}\Delta_b^{\alpha} y\right)(th)\right)h
\longrightarrow \min\, ,\\
(y(a)=A),\ (y(b)=B),
\end{gathered}
\end{equation}
as well as transversality conditions when
the boundary conditions $y(a)=A$ or $y(b)=B$ are not given
(see Section~\ref{sec2} for definitions and notations).
The main result of \cite{BFT:SP} gives an Euler--Lagrange type equation
for problem \eqref{eq:prb:int}, but no clues are devised
for the solution of such fractional $h$-difference equations.
Instead, some examples are solved numerically \cite{BFT,BFT:SP}.
Here we develop further the subject of the calculus of variations within the
fractional discrete setting, by obtaining explicit solutions
to the fractional difference Euler--Lagrange equations \cite{BFT,BFT:SP}.

Our results are given in Section~\ref{sec2},
where we prove some new formulas for the
fractional $h$-difference operator. The obtained results
are then used in Section~\ref{sec3} to solve
two illustrative examples of \eqref{eq:prb:int},
for which the global minimizers are explicitly found in exact form.
This is in contrast with \cite{BFT,BFT:SP}, where all the solutions
are obtained via approximated numerical computations.


\section{Main Results}
\label{sec2}

Before stating and proving our results, we introduce some
definitions and notations. Let $h>0$ and put
$\mathbb{T}=\{a,a+h,a+2h,\ldots,b\}$ with
$a\in\mathbb{R}$ and $b=a+kh$ for $k\in\{2,3,\ldots\}$. Let us
denote by $\mathcal{F}_\mathbb{T}$ the set of real valued functions
defined on $\mathbb{T}$, $\sigma_h(t)=t+h$, and $\rho_h(t)=t-h$.

\begin{defin}
For a function $f\in\mathcal{F}_\mathbb{T}$
the forward $h$-difference operator is defined as
$$
(\Delta_h f)(t)=\frac{f(\sigma_h(t))-f(t)}{h} \, ,
\quad t\in\{a,a+h,a+2h,\ldots,\rho_h(b)\},
$$
while the $h$-difference sum is given by
$$
({_a}\Delta^{-1}_hf)(t)=\sum_{k=\frac{a}{h}}^{\frac{t}{h}-1}f(kh)h\, ,
\quad t\in\{a,a+h,a+2h,\ldots,\sigma_h(b)\}.
$$
\end{defin}

\begin{defin}
For arbitrary $x,y\in\mathbb{R}$ the $h$-factorial function is defined by
\begin{equation*}
x_h^{(y)}=h^y\frac{\Gamma(\frac{x}{h}+1)}{\Gamma(\frac{x}{h}+1-y)}\, ,
\end{equation*}
where $\Gamma$ is the Euler gamma function.
We use the convention that division at a pole yields zero.
\end{defin}


In \cite{BFT:SP} it is remarked that in the
case $h = 1$, then $x_h^{(y)}$ coincides with the falling factorial power.
One also expects to see that $x_h^{(y)}$ converges to $x^y$ when $h$ tends to zero.
Since this is not addressed in \cite{BFT:SP}, we prove it here.

\begin{prop}
\label{prop:ref:asked}
For $x\geq 0$ and $y\in\mathbb{R}$,
\begin{equation}
\label{lim0}
\lim_{h\rightarrow 0} x_h^{(y)}=x^y.
\end{equation}
\end{prop}

\begin{proof}
Equality \eqref{lim0} is a straightforward consequence
of the following  well-known asymptotic formula for the Gamma
function:
$$
\lim_{x \rightarrow +\infty} \frac{\Gamma(x + \beta)}{\Gamma(x)}
= (x+\beta-1)^\beta, \quad \beta \in \mathbb{R}
$$
(see, \textrm{e.g.},
inequality (33) and Corollary~3 in \cite{MR1676733}).
Indeed, starting from the definition of $x_h^{(y)}$
and introducing the new variable $t = x/h + 1 - y$, we have
$$
x_h^{(y)}=\frac{x^y}{(t+y-1)^y} \cdot \frac{\Gamma(t+y)}{\Gamma(t)}
$$
for any $y \in \mathbb{R}$. We obtain \eqref{lim0}
taking the limit $t \rightarrow +\infty$ or,
equivalently, the limit $h\rightarrow 0^+$.
\end{proof}

The motivation for the next definition can be found in \cite{BFT:SP}.

\begin{defin}[\cite{BFT:SP}]
\label{def0}
Let $f\in\mathcal{F}_\mathbb{T}$. The \emph{left and right fractional
$h$-sum} of order $\nu>0$ are, respectively, the operators
$_a\Delta_h^{-\nu} : \mathcal{F}_\mathbb{T} \rightarrow
\mathcal{F}_{\tilde{\mathbb{T}}_\nu^+}$ and $_h\Delta_b^{-\nu} :
\mathcal{F}_\mathbb{T} \rightarrow
\mathcal{F}_{\tilde{\mathbb{T}}_\nu^-}$,
$\tilde{\mathbb{T}}_\nu^\pm = \{a \pm \nu h,a \pm \nu h\pm
h,\ldots,b \pm \nu h\}$, given by
\begin{equation*}
\begin{split}
\left({_a}\Delta_h^{-\nu}f\right)(t)
&=\frac{1}{\Gamma(\nu)}\sum_{k=\frac{a}{h}}^{\frac{t}{h}-\nu}(t-\sigma_h(kh))_h^{(\nu-1)}f(kh)h,\\
\left({_h}\Delta_b^{-\nu}f\right)(t)
&=\frac{1}{\Gamma(\nu)}\sum_{k=\frac{t}{h}+\nu}^{\frac{b}{h}}(kh-\sigma_h(t))_h^{(\nu-1)}f(kh)h.
\end{split}
\end{equation*}
We define $\left({_a}\Delta_h^{0}f\right)(t)=f(t)$
and $\left({_h}\Delta_b^{0}f\right)(t)=f(t)$.
\end{defin}

\begin{defin}
\label{def1}
Let $0<\alpha\leq 1$ and set $\gamma = 1-\alpha$.
The \emph{left fractional $h$-difference}
$_a\Delta_h^\alpha f$ and the
\emph{right fractional $h$-difference}
$_h\Delta_b^\alpha f$ of order $\alpha$
of a function $f\in\mathcal{F}_\mathbb{T}$
are defined, respectively, by
\begin{equation*}
\begin{split}
\left(_a\Delta_h^\alpha f\right)(t)
&= \left(\Delta_h{_a}\Delta_h^{-\gamma}f\right)(t)\, ,
\quad t\in\{a + \gamma h,a + \gamma h + h,\ldots,\rho_h(b) + \gamma h\}\, ,\\
\left({_h}\Delta_b^\alpha f\right)(t)
&=-(\Delta_h{_h}\Delta_b^{-\gamma}f(t)\, ,
\quad t\in\{a - \gamma h,a - \gamma h - h,\ldots,\rho_h(b) - \gamma h\}.
\end{split}
\end{equation*}
\end{defin}

\begin{rem}
We define fractional sums/differences for functions on a bounded domain.
This is done so, because of the problems of the calculus
of variations we consider here. Nevertheless, one can use our definitions
for functions with unbounded domains:
an unbounded domain from above for the left fractional sum/difference,
an unbounded domain from below for the right fractional sum/difference.
\end{rem}

Let us now recall a result that will be used later in finding
solutions to the boundary value problems originated from the
fractional $h$-difference calculus of variations.

\begin{thm}[Theorem~2.10 of \cite{BFT:SP}]
\label{thm1}
Let $f\in\mathcal{F}_\mathbb{T}$ and $\nu\geq0$. Then,
\begin{equation*}
\left({_a}\Delta_{h}^{-\nu}
\Delta_hf\right)(t)=\left(\Delta_h{_a}\Delta_h^{-\nu}f\right)(t)-\frac{\nu}{\Gamma(\nu
+ 1)}(t-a)_h^{(\nu-1)}f(a)
\end{equation*}
for all $t\in\{a+\nu h,a+\nu h+h,\ldots,\rho_h(b)+\nu h\}$.
\end{thm}

The next lemma permits to shorten the proofs of our main
results. Essentially, it allow us to borrow information from the
formulas obtained in \cite{Atici0}.

\begin{lem}
\label{lem0}
Let $f\in\mathcal{F}_\mathbb{T}$ and $\nu\geq0$. Then,
\begin{equation*}
\left(_a\Delta_{h}^{-\nu}f\right)(t)
=h^\nu\left({_\frac{a}{h}}\Delta_{1}^{-\nu} s \mapsto f(sh)\right)\left(\frac{t}{h}\right),
\quad t\in\{a+\nu h, a+\nu h+h,\ldots,b+\nu h\}.
\end{equation*}
\end{lem}

\begin{proof}
We have
\begin{align*}
\left(_a\Delta_{h}^{-\nu}f\right)(t)
&=\frac{1}{\Gamma(\nu)}\sum_{k=\frac{a}{h}}^{\frac{t}{h}-\nu}(t-\sigma_h(kh))_h^{(\nu-1)}f(kh)h\\
&=\frac{1}{\Gamma(\nu)}\sum_{k=\frac{a}{h}}^{t'-\nu}(t'h-\sigma_h(kh))_h^{(\nu-1)}f(kh)h\\
&=\frac{1}{\Gamma(\nu)}\sum_{k=\frac{a}{h}}^{t'-\nu}h^{\nu-1}(t'-\sigma_1(k))_1^{(\nu-1)}f(kh)h\\
&=h^\nu\left({_\frac{a}{h}}\Delta_{1}^{-\nu} s \mapsto f(sh)\right)\left(\frac{t}{h}\right),
\end{align*}
and therefore the proof is done.
\end{proof}

\begin{lem}[Lemma~2.3 of \cite{Atici0}]
\label{teo0}
Let $\mu,\nu$ be two real numbers such that
$\mu,\mu+\nu\in\mathbb{R}\backslash\{\ldots,-2,-1\}$. Then,
\begin{equation*}
\left(_\mu\Delta_1^{-\nu} s \mapsto s_1^{(\mu)}\right)(t)
=\frac{\Gamma(\mu+1)}{\Gamma(\mu+\nu+1)}t_1^{(\mu+\nu)},
\quad t=\mu+\nu, \mu+\nu+1, \ldots
\end{equation*}
\end{lem}

\begin{cor}
\label{cor0}
Suppose that
$\frac{\mu}{h},\frac{\mu}{h}+\nu\in\mathbb{R}\backslash\{\ldots,-2,-1\}$.
Then,
\begin{equation*}
\left(_a\Delta_{h}^{-\nu} s \mapsto \left(s-a+\mu\right)_h^{(\frac{\mu}{h})}\right)(t)
=\frac{\Gamma(\frac{\mu}{h}+1)}{\Gamma(\frac{\mu}{h}
+\nu+1)}\left(t-a+\mu\right)_h^{(\frac{\mu}{h}+\nu)}
\end{equation*}
for all $t\in\{a+\nu h, a+\nu h+h,\ldots\}$.
\end{cor}

\begin{proof}
The result is a simple consequence of Lemma~\ref{lem0} and Lemma~\ref{teo0}.
\end{proof}

\begin{cor}
\label{cor1}
Suppose that
$-\frac{\mu}{h},-\frac{\mu}{h}+\nu\in\mathbb{R}\backslash\{\ldots,-2,-1\}$.
Then,
\begin{equation*}
\left(_h\Delta_{b}^{-\nu} s \mapsto
\left(b-\mu-s\right)_h^{(-\frac{\mu}{h})}\right)(t)
=\frac{\Gamma(-\frac{\mu}{h}+1)}{\Gamma(-\frac{\mu}{h}
+\nu+1)}\left(b-\mu-t\right)_h^{(-\frac{\mu}{h}+\nu)}
\end{equation*}
for all $t\in\{b-\nu h, b-\nu h-h,\ldots\}$.
\end{cor}

\begin{proof}
Using Corollary~\ref{cor0}, we get
\begin{align*}
\biggl({_h\Delta_{b}}^{-\nu}& s \mapsto
\left(b-\mu-s\right)_h^{(-\frac{\mu}{h})}\biggr)(t)
=\frac{1}{\Gamma(\nu)}\sum_{k=\frac{t}{h}+\nu}^{\frac{b}{h}}(kh
-\sigma_h(t))_h^{(\nu-1)}\left(b-\mu-kh\right)_h^{(-\frac{\mu}{h})}h\\
&=\frac{1}{\Gamma(\nu)}\sum_{k=\frac{t-b+\mu}{h}+\nu}^{\frac{\mu}{h}}(kh+b
-\mu-\sigma_h(t))_h^{(\nu-1)}\left(-kh\right)_h^{(-\frac{\mu}{h})}h\\
&=\frac{1}{\Gamma(\nu)}\sum_{k=-\frac{\mu}{h}}^{\frac{b-\mu-t}{h}
-\nu}(b-\mu-t-\sigma_h(kh))_h^{(\nu-1)}\left(kh\right)_h^{(-\frac{\mu}{h})}h\\
&=\left(_{-\mu}\Delta_{h}^{-\nu} s \mapsto s_h^{(-\frac{\mu}{h})}\right)(b-\mu-t)
=\frac{\Gamma(-\frac{\mu}{h}+1)}{\Gamma(-\frac{\mu}{h}
+\nu+1)}\left(b-\mu-t\right)_h^{(-\frac{\mu}{h}+\nu)}
\end{align*}
and the proof is complete.
\end{proof}

We now state and prove the law of exponents
for the fractional $h$-difference sums.

\begin{thm}
\label{thm0}
Let $f\in\mathcal{F}_\mathbb{T}$ and $\mu,\nu\geq0$. Then,
\begin{equation}
\label{ponto1}
\left(_{a+\nu h}\Delta_h^{-\mu}{_a}\Delta_h^{-\nu}f\right)(t)
=\left({_a}\Delta_h^{-(\mu+\nu)}f\right)(t),
\end{equation}
where $t\in\{a+(\mu+\nu)h,a+(\mu+\nu)h+h,\ldots,b+(\mu+\nu)h\}$; and
\begin{equation}
\label{ponto2}
\left(_{h}\Delta_{b-\nu h}^{-\mu}{_h}\Delta_b^{-\nu}f\right)(t)
=\left({_h}\Delta_b^{-(\mu+\nu)}f\right)(t),
\end{equation}
where $t\in\{b-(\mu+\nu)h,b-(\mu+\nu)h-h,\ldots,a-(\mu+\nu)h\}$.
\end{thm}

\begin{proof}
We prove \eqref{ponto1} only, \eqref{ponto2} being
accomplished analogously. First, note that if $\nu=0$ or
$\mu=\nu=0$, then the equality is valid by definition.
Therefore, assume that
$\nu-1,\nu-1+\mu\in\mathbb{R}\backslash\{\ldots,-2,-1\}$. Then,
\begin{align*}
\biggl({_{a+\nu h}}&\Delta_h^{-\mu}{{_a}\Delta_h^{-\nu}}f\biggr)(t)\\
&=\frac{h^2}{\Gamma(\nu)\Gamma(\mu)}\sum_{s=\frac{a}{h}+\nu}^{\frac{t}{h}
-\mu}(t-\sigma_h(sh))_h^{(\mu-1)}\sum_{r=\frac{a}{h}}^{s
-\nu}(sh-\sigma_h(rh))_h^{(\nu-1)}f(rh)\\
&=\frac{h^2}{\Gamma(\nu)\Gamma(\mu)}\sum_{r=\frac{a}{h}}^{\frac{t}{h}
-(\mu+\nu)}\sum_{s=r+\nu}^{\frac{t}{h}-\mu}(t-\sigma_h(sh))_h^{(\mu
-1)}(sh-\sigma_h(rh))_h^{(\nu-1)}f(rh)\\
&=\frac{h^2}{\Gamma(\nu)\Gamma(\mu)}\sum_{r=\frac{a}{h}}^{\frac{t}{h}
-(\mu+\nu)}\sum_{s=\nu-1}^{\frac{t}{h}-r-1-\mu}(t
-\sigma_h[(s+r+1)h])_h^{(\mu-1)}(sh)_h^{(\nu-1)}f(rh)\\
&=\frac{h^2}{\Gamma(\nu)\Gamma(\mu)}\sum_{r=\frac{a}{h}}^{\frac{t}{h}
-(\mu+\nu)}\sum_{s=\nu-1}^{\frac{t-\sigma_h(rh)}{h}-\mu}(t-\sigma_h(rh)
-\sigma_h(sh))_h^{(\mu-1)}(sh)_h^{(\nu-1)}f(rh)\\
&=\frac{h}{\Gamma(\nu)}\sum_{r=\frac{a}{h}}^{\frac{t}{h}-(\mu+\nu)}\left(_{(\nu
-1)h}\Delta_h^{-\mu} s \mapsto s^{(\nu-1)}_h\right)(t-\sigma_h(rh))f(rh)
\end{align*}
\begin{align*}
&=\frac{h}{\Gamma(\nu)}\sum_{r=\frac{a}{h}}^{\frac{t}{h}-(\mu
+\nu)}\frac{\Gamma(\nu)}{\Gamma(\nu+\mu)}\left(t
-\sigma_h(rh)\right)_h^{(\nu+\mu-1)}f(rh)
=\left({_a}\Delta_h^{-(\mu+\nu)}f\right)(t),
\end{align*}
which shows the intended equality.
\end{proof}

The next theorem is crucial in order to solve some
fractional $h$-difference Euler--Lagrange
equations (see the examples in Section~\ref{sec3}).

\begin{thm}
\label{thm2}
Let $0<\alpha\leq 1$ and $f\in\mathcal{F}_\mathbb{T}$. Then,
\begin{equation}
\label{eq0}
\left(_h\Delta_b^\alpha f\right)(t)=0,\quad
t\in\{a-(1-\alpha)h,a-(1-\alpha)h+h,\ldots,\rho_h(b)-(1-\alpha)h\},
\end{equation}
if and only if
\begin{equation}
\label{eq1}
f(t)=\frac{c}{\Gamma({\alpha})}(b-(1-\alpha)h-t)_h^{(\alpha-1)},
\quad t\in\{a,a+h,\ldots,b\},
\end{equation}
where $c$ is an arbitrary constant.
\end{thm}

\begin{proof}
If $f$ is given as in \eqref{eq1}, we immediately get \eqref{eq0}
using Corollary~\ref{cor1}. Suppose now that equality in
\eqref{eq0} holds in the mentioned domain. Then, by definition of
fractional difference,
\begin{equation}
\label{eq2}
\left(_h\Delta_b^{-(1-\alpha)}f\right)(t)=c,\quad
t\in\{a-(1-\alpha)h,a+h-(1-\alpha)h,\ldots,b-(1-\alpha)h\},\quad
c\in\mathbb{R}.
\end{equation}
Applying the operator $_h\Delta_{b-(1-\alpha)h}^{-\alpha}$ to both
sides of equality in \eqref{eq2}, and using \eqref{ponto2} of
Theorem~\ref{thm0}, we get
$\left(_h\Delta_{b}^{-1}f\right)(t)
=c\left(_h\Delta_{b-(1-\alpha)h}^{-\alpha}1\right)(t)$,
$t\in\{a-h,a,\ldots,b-h\}$. Corollary~\ref{cor1} now implies that
\begin{equation}
\label{eq3}
\left(_h\Delta_{b}^{-1}f\right)(t)
=\frac{c}{\Gamma({\alpha+1})}(b-(1-\alpha)h-t)_h^{(\alpha)},
\quad t\in\{a-h,a,\ldots,b-h\}.
\end{equation}
An application of the operator $\Delta_h$ to both sides of the
equality in \eqref{eq3} gives
$$
f(\sigma_h(t))
=c\frac{\alpha}{\Gamma({\alpha+1})}(b
-(1-\alpha)h-\sigma_h(t))_h^{(\alpha-1)},
\quad t\in\{a-h,a,\ldots,b-2h\},
$$
or
\begin{equation}
\label{eq4}
f(t)=\frac{c}{\Gamma({\alpha})}(b-(1-\alpha)h-t)_h^{(\alpha-1)},
\quad t\in\{a,a+h,\ldots,b-h\}.
\end{equation}
Setting $t=b-h$ in \eqref{eq3} we get $f(b)h=ch^\alpha$, \textrm{i.e.},
equality in \eqref{eq4} is also valid when $t=b$.
\end{proof}

\begin{rem}
\label{rem1}
Similar steps as those done in the proof of Theorem~\ref{thm2}
permit us to prove the following equivalence: for $0<\alpha\leq
1$, $c\in\mathbb{R}$, and $f\in\mathcal{F}_\mathbb{T}$, we have
\begin{equation*}
\left(_h\Delta_b^\alpha f\right)(t)=c,
\quad t\in\{a-(1-\alpha)h,a+h-(1-\alpha)h,\ldots,b-h-(1-\alpha)h\}
\end{equation*}
if and only if
\begin{equation*}
f(t)=\frac{c(b+\alpha
h-b\alpha-\alpha^2h-t)+d\alpha}{\Gamma({\alpha}+1)}(b-(1-\alpha)h-t)_h^{(\alpha-1)},
\quad t\in\{a,a+h,\ldots,b\},
\end{equation*}
where $d$ is an arbitrary constant. Indeed, from Corollary~\ref{cor1}
and Theorem~\ref{thm0}, we have
\begin{equation*}
\begin{split}
&\bigl({_h\Delta_b^\alpha} f\bigr)(t)=c\Leftrightarrow
\bigl({_h\Delta_b}^{-(1-\alpha)} f\bigr)(t)=-ct+d\\
&\Leftrightarrow \left(_h\Delta_b^{-(1-\alpha)} f\right)(t)=c(b+\alpha h-t-b-\alpha h)+d\\
&\Leftrightarrow \left(_h\Delta_b^{-(1-\alpha)} f\right)(t)=c(b+\alpha h-t)-c(b+\alpha h)+d\\
&\Leftrightarrow \left({_h}\Delta_{b-(1-\alpha)h}^{-\alpha}{_h}\Delta_b^{-(1-\alpha)} f\right)(t)
=c\left({_h}\Delta_{b-(1-\alpha)h}^{-\alpha} s \mapsto b+\alpha h-s\right)(t)\\
&\qquad -[c(b+\alpha h)-d]\left({_h}\Delta_{b-(1-\alpha)h}^{-\alpha}1\right)(t)\\
&\Leftrightarrow \left({_h}\Delta_{b}^{-1}
f\right)(t)=\frac{c}{\Gamma{(\alpha+2)}}(b+\alpha
h-t)^{(\alpha+1)}_h-\frac{c(b+\alpha
h)-d}{\Gamma{(\alpha+1)}}(b-(1-\alpha)
h-t)^{(\alpha)}_h\\
&\Leftrightarrow f(t)=\frac{c}{\Gamma{(\alpha+1)}}(b+\alpha
h-t)^{(\alpha)}_h-\frac{c(b+\alpha
h)-d}{\Gamma{(\alpha)}}(b-(1-\alpha)
h-t)^{(\alpha-1)}_h\\
&\Leftrightarrow
f(t)=\frac{c}{\Gamma{(\alpha+1)}}h^\alpha\frac{\Gamma\left(\frac{b+\alpha
h-t}{h}+1\right)}{\Gamma\left(\frac{b+\alpha
h-t}{h}+1-\alpha\right)}-\frac{c(b+\alpha
h)-d}{\Gamma{(\alpha)}}(b-(1-\alpha)
h-t)^{(\alpha-1)}_h\\
&\Leftrightarrow f(t)=\frac{c}{\Gamma{(\alpha+1)}}(b+\alpha
h-t)h^{\alpha-1}\frac{\Gamma\left(\frac{b-(1-\alpha)
h-t}{h}+1\right)}{\Gamma\left(\frac{b-(1-\alpha)
h-t}{h}+1-(\alpha-1)\right)}\\
&\qquad -\frac{c(b+\alpha h)-d}{\Gamma{(\alpha)}}(b-(1-\alpha)
h-t)^{(\alpha-1)}_h\\
&\Leftrightarrow f(t)=\frac{c(b+\alpha
h-b\alpha-\alpha^2h-t)+d\alpha}{\Gamma({\alpha}+1)}(b-(1-\alpha)h-t)_h^{(\alpha-1)}.
\end{split}
\end{equation*}
\end{rem}

We end this section enunciating the analogue of Theorem~\ref{thm2}
for the left fractional $h$-difference.

\begin{thm}
\label{thm4}
Let $0<\alpha\leq 1$ and $f\in\mathcal{F}_\mathbb{T}$. Then,
\begin{equation*}
\left(_a\Delta_h^\alpha f\right)(t)=0,
\quad t\in\{a+(1-\alpha)h,a+(1-\alpha)h+h,\ldots,\rho_h(b)+(1-\alpha)h\},
\end{equation*}
if and only if
\begin{equation*}
f(t)=\frac{c}{\Gamma({\alpha})}(t-(1-\alpha)h-a)_h^{(\alpha-1)},
\quad t\in\{a,a+h,\ldots,b\},
\end{equation*}
where $c$ is an arbitrary constant.
\end{thm}

\begin{proof}
The proof is analogous to the one of Theorem~\ref{thm2}.
\end{proof}


\section{Applications to the Calculus of Variations}
\label{sec3}

We now give two examples of application of our results.
The main achievement is to obtain explicit solutions for
some problems of the calculus of variations.
In this section we omit the subscript $h$ in $\sigma_h$ and $\rho_h$.
For convenience of notation we write $y^\sigma(t) = y(\sigma(t))$.

\begin{ex}
\label{eq:ex:1}
Let us consider the following data: let $a\in\mathbb{R}$, $h > 0$,
$b=a+kh$ with $k\in\{2,3,\ldots\}$, and $0<\alpha\leq 1$.
Moreover, let $A$ and $B$ be to given real numbers. We want to
find a function $y\in\mathcal{F}_\mathbb{T}$ that solves the problem
\begin{equation}
\label{EL:ex1}
\mathcal{L}(y)
=\sum_{t=\frac{a}{h}}^{\frac{b}{h}-1} \left({_a}\Delta_h^{\alpha}
y\right)^2(th)h \longrightarrow \min \, ,
\quad y(a)=A \, , \quad y(b)=B \, .
\end{equation}
By \cite[Theorem~3.5]{BFT:SP} we have that if $\hat{y}$ is a minimizer
of $\mathcal{L}$ given in \eqref{EL:ex1}, then
\begin{equation}
\label{eq5}
\left(_h\Delta_{\rho(b)}^\alpha{_a}\Delta_{h}^\alpha\hat{y}\right)(t)=0,
\quad t\in\{a,a+h,\ldots,b-2h\}.
\end{equation}

\begin{rem}
\label{rem0}
At a first glance the sum in \eqref{EL:ex1} and the equation in
\eqref{eq5} seem to be meaningless due to the possible values of
the variable $t$. However, they aren't by the fact that the
authors in \cite{BFT:SP} used the following notation for the
difference operators:
\begin{align*}
\left({_a}\Delta_{h}^\alpha f\right)(t)&=\left(_a\Delta_{h}^\alpha
f\right)(t+(1-\alpha)h),\quad t\in\{a,a+h,\ldots,b-h\},\\
\left({_h}\Delta_{b}^\alpha f\right)(t)&=\left(_h\Delta_{b}^\alpha
f\right)(t-(1-\alpha)h),\quad t\in\{a,a+h,\ldots,b-h\}.
\end{align*}
\end{rem}
An application of our Theorem~\ref{thm2} to the equality in \eqref{eq5} gives
$$
\left({_a}\Delta_{h}^\alpha\hat{y}\right)(t)
=\frac{c}{\Gamma({\alpha})}(\rho(b)-(1-\alpha)h-t)_h^{(\alpha-1)},
\quad t\in\{a,a+h,\ldots,\rho(b)\},
$$
with $c\in\mathbb{R}$ or
$$
\left(\Delta_h{_a}\Delta_{h}^{1-\alpha}\hat{y}\right)(t)
=\frac{c}{\Gamma({\alpha})}(\rho(b)-(t+(1-\alpha)h))_h^{(\alpha-1)},
\quad t\in\{a,a+h,\ldots,\rho(b)\}.
$$
We now remember Remark~\ref{rem0} and apply the operator
$_{a+(1-\alpha)h}\Delta^{-\alpha}_h$ to both
sides of this equality. From Theorems~\ref{thm1} and \ref{thm0} it follows that
\begin{multline}
\label{eq6}
\left(_{a+(1-\alpha)h}\Delta^{-\alpha}_h\Delta_h{_a}\Delta_{h}^{1-\alpha}\hat{y}\right)(t)
=\frac{c}{\Gamma({\alpha})}\left(
{_{a+(1-\alpha)h}}\Delta^{-\alpha}_h  s \mapsto (\rho(b)-s)_h^{(\alpha-1)}\right)(t)\\
\Leftrightarrow \hat{y}(t)=\frac{c}{\Gamma({\alpha})}\left({_{a
+(1-\alpha)h}}\Delta^{-\alpha}_h s \mapsto (\rho(b)-s)_h^{(\alpha-1)}\right)(t)\\
+\frac{1}{\Gamma(\alpha)}(t-(a+(1-\alpha)h))_h^{(\alpha-1)}\hat{y}(a),
\end{multline}
with $t\in\{a+h,a+2h,\ldots,b\}$.
The constant $c$ is determined by the end condition $\hat{y}(b)=B$.

\begin{rem}
We point out that if $\alpha=1$ we get the ``straight line''
connecting the points $(a,A)$ and $(b,B)$ as the solution of the
Euler--Lagrange equation \eqref{eq6}, \textrm{i.e.},
$\hat{y}(t)=\frac{B-A}{b-a}(t-a)+A$. This result
can be found, \textrm{e.g.}, in \cite{Bohner,Gusein}.
\end{rem}

We now show that the function $\hat{y}$ given by
\eqref{eq6} furnishes in fact a global minimum to the problem
\eqref{EL:ex1}. To do that, we recall the fractional
$h$-summation by parts formula obtained by the authors in \cite{BFT:SP}
(we continue to use here the notation mentioned in Remark~\ref{rem0}).

\begin{thm}[Theorem~3.2 of \cite{BFT:SP}]
\label{thm3}
Let $f$ and $g$ be real valued functions defined on $\{a,a+h,\ldots,b-h\}$
and $\{a,a+h,\ldots,b\}$, respectively. Fix $0<\alpha\leq 1$ and put
$\gamma= 1-\alpha$. Then,
\begin{multline*}
\sum_{t=\frac{a}{h}}^{\frac{b}{h}-1}f(th)\left(_a\Delta_h^\alpha
g\right)(th)h=h^\gamma f(\rho(b))g(b)-h^\gamma
f(a)g(a)+\sum_{t=\frac{a}{h}}^{\frac{b}{h}-2}\left(_h\Delta_{\rho(b)}^\alpha
f\right)(th)g^\sigma(th)h\\
+\frac{\gamma \, g(a)}{\Gamma(\gamma+1)}\left(\sum_{t=\frac{a}{h}}^{\frac{b}{h}-1}(th+\gamma
h-a)_h^{(\gamma-1)}f(th)h-\sum_{t=\frac{\sigma(a)}{h}}^{\frac{b}{h}-1}(th+\gamma
h-\sigma(a))_h^{(\gamma-1)}f(th)h\right).
\end{multline*}
\end{thm}

Before proceeding, we need the following definition:

\begin{defin}
We say that a Lagrangian
$L(t,u,v):\mathbb{T}\times\mathbb{R}^2\rightarrow\mathbb{R}$
is jointly convex in $(u,v)$ if
$$
L(t,u,v)-L(t,u',v')\geq(u-u')L_u(t,u',v')+(v-v')L_v(t,u',v'),
\quad u,u',v,v'\in\mathbb{R},
$$
provided the partial derivatives $L_u$ and $L_v$ exist.
\end{defin}

We are now able to prove the following theorem.

\begin{thm}
Consider the set $S=\{f\in\mathcal{F}_\mathbb{T}:f(a)=A,f(b)=B\}$.
Suppose that the Lagrangian
$L(t,u,v):\mathbb{T}\times\mathbb{R}^2\rightarrow\mathbb{R}$
of the minimization problem
$$
\mathcal{L}(y)=
\sum_{t=\frac{a}{h}}^{\frac{b}{h}-1}
L(th,y^\sigma(th),\left({_a}\Delta_h^{\alpha} y\right)(th))h
\longrightarrow \min \, , \quad y(a)=A \, , \quad y(b)=B,
$$
is jointly convex in $(u,v)$. Assume that the function $\hat{y}\in S$
satisfies the Euler--Lagrange equation for this problem, \textrm{i.e.},
\begin{equation}
\label{eq7}
L_u[\hat{y}](t)+\left(_h\Delta_{\rho(b)}^\alpha
L_v[\hat{y}]\right)(t)=0,\quad t\in\{a,a+h,\ldots,b-2h\},
\end{equation}
where $[y](s)=(s,y^\sigma(s),\left({_a}\Delta_h^{\alpha} y\right)(s)))$.
Then, $\hat{y}$ furnishes a global minimum to $\mathcal{L}$ in the set $S$.
\end{thm}

\begin{proof}
Let $y\in S$ be an arbitrary function. Suppose that $\hat{y}\in S$
satisfies equation in \eqref{eq7}. Since $L(t,u,v)$ is jointly
convex in $(u,v)$, we get, with the use of Theorem~\ref{thm3},
\begin{align*}
\sum_{t=\frac{a}{h}}^{\frac{b}{h}-1}\left\{L[y]-L[\hat{y}]\right\}h
&\geq\sum_{t=\frac{a}{h}}^{\frac{b}{h}-1}\left\{(y^\sigma-\hat{y}^\sigma)L_u[\hat{y}]
+(_a\Delta_{h}^\alpha y-{_a}\Delta_{h}^\alpha\hat{y})L_v[\hat{y}]\right\}h\\
&=\sum_{t=\frac{a}{h}}^{\frac{b}{h}-2}\left\{(y^\sigma-\hat{y}^\sigma)\left(L_u[\hat{y}]
+{_h}\Delta_{\rho(b)}^\alpha L_v[\hat{y}]\right)\right\}h\\
&=0.
\end{align*}
The theorem is proved.
\end{proof}

It is clear that the Lagrangian $L(t,u,v)=v^2$ in \eqref{EL:ex1}
is jointly convex in $(u,v)$. Therefore, the function $\hat{y}$
defined in \eqref{eq6} furnishes a global minimum to \eqref{EL:ex1}.
\end{ex}

We end solving another fractional
difference problem of the calculus of variations.

\begin{ex}
Let $a\in\mathbb{R}$, $h > 0$, $b=a+kh$ with $k\in\{2,3,\ldots\}$, $0<\alpha\leq 1$,
and $A$ and $B$ be two given real numbers. We consider the following
variational problem:
\begin{equation}
\label{eq:ex:2}
\mathcal{L}(y)= \sum_{t=\frac{a}{h}}^{\frac{b}{h}-1}
\left[\frac{1}{2}\left({_a}\Delta_h^{\alpha}
y\right)^2(th)-y^\sigma(th)\right]h \longrightarrow \min \, ,
\quad y(a)=A \, , \quad y(b)=B \, .
\end{equation}
\end{ex}

The Euler--Lagrange equation for problem \eqref{eq:ex:2} is
\begin{equation}
\label{eq8}
\left(_h\Delta_{\rho(b)}^\alpha{_a}\Delta_{h}^\alpha
y\right)(t)=1,\quad t\in\{a,a+h,\ldots,b-2h\}.
\end{equation}
In view of Remark~\ref{rem1}, we get from equality in \eqref{eq8} that
$$
\left({_a}\Delta_{h}^\alpha
y\right)(t)=\frac{\rho(b)+\alpha
h-\rho(b)\alpha-\alpha^2h-t
+d\alpha}{\Gamma({\alpha}+1)}(\rho(b)-(1-\alpha)h-t)_h^{(\alpha-1)},
$$
for a constant $d\in\mathbb{R}$ to be determined.
Following the same steps as those done for Example~\ref{eq:ex:1}, we get
\begin{multline}
\label{eq:sol:ex:2}
\hat{y}(t)=\left(_{a+(1-\alpha)h}\Delta^{-\alpha}_h
s \mapsto \frac{b-\rho(b)\alpha
-\alpha^2h-s+d\alpha}{\Gamma({\alpha}+1)}(\rho(b)-s)_h^{(\alpha-1)}\right)(t)\\
+\frac{1}{\Gamma(\alpha)}(t-(a+(1-\alpha)h)_h^{(\alpha-1)}y(a),
\end{multline}
for $t\in\{a+h,a+2h,\ldots,b\}$. Finally, we show that the
Lagrangian $L(t,u,v)=v^2-u$ is jointly convex in $(u,v)$.
Indeed, for $u,v,u',v'\in\mathbb{R}$ we have
$$
v^2-u-(v'^2-u')\geq-(u-u')+(v-v')2v'\Leftrightarrow (v-v')^2\geq 0.
$$
We conclude that $\hat{y}$ given by \eqref{eq:sol:ex:2}
is the global minimizer of \eqref{eq:ex:2}.


\section*{Acknowledgments}

The authors were supported by the
\emph{Portuguese Foundation for Science and Technology} (FCT)
through the \emph{Center for Research and Development
in Mathematics and Applications} (CIDMA).
They are very grateful to a referee
for valuable remarks and comments, which
significantly contributed to the quality of the paper.



\end{document}